\documentclass[12pt]{amsart}

 \usepackage{graphicx,mdframed}
\usepackage{amsmath,amsthm,amsfonts,amscd,amssymb,mathrsfs,hyperref}
\usepackage[all]{xypic}
\usepackage[enableskew]{youngtab} % For young tableaux

\usepackage{color}
\usepackage{amsmath}
\usepackage{epsfig}
\usepackage{epstopdf}
\usepackage[linesnumbered,ruled,vlined]{algorithm2e}

\usepackage[usenames,dvipsnames]{xcolor}
 
\usepackage{tikz}
\usepackage{pgfplots}
\usepackage{braids}

\newtheorem{theorem}{Theorem}[section]

\newtheorem{lemma}[theorem]{Lemma}

\newtheorem{prop}[theorem]{Proposition}

\def\cB{{\mathcal B}}
\def\cT{{\mathcal T}}
\def\cC{{\mathcal C}}

\newcommand{\Path}{{P}}

\newcommand{\mfS}{\mathfrak{S}}

\newcommand{\Ldelta}{L\left(\delta',\delta''\right)}

\newcommand{\branch}{\mathcal{B}}
\newcommand{\braidSymbol}{\mathbf{B}}

\newcommand{\piLeft}{\pi^{-1}_{-}}
\newcommand{\piRight}{\pi^{-1}_{+}}

\theoremstyle{definition}
\newtheorem{definition}[theorem]{Definition}
\newtheorem{example}[theorem]{Example}

\theoremstyle{remark}
\newtheorem{remark}[theorem]{Remark}

\DeclareMathOperator{\Cross}{Cross}

\begin{document}

\title{Numerical computation of braid groups}

%\author{Jose Israel Rodriguez and Botong Wang}
\author[Rodriguez]{Jose Israel Rodriguez}
\address{
Jose Israel Rodriguez\\
The University of Chicago\\
Dept. of Statistics\\
5734 S. University Ave.
Chicago, IL 60637}
\email{joisro@Uchicago.edu}
\urladdr{http://home.uchicago.edu/~joisro}

\author[Wang]{Botong Wang}
\address{
Botong Wang\\
University of Wisconsin-Madison\\
 Department of Mathematics\\
Van Vleck Hall, 480 Lincoln Drive, 
Madison, WI}
\email{wang@math.wisc.edu}
\urladdr{http://www.math.wisc.edu/~wang/}

\begin{abstract}
In this article, we give a numerical algorithm to compute braid
groups of curves, hyperplane arrangements, and parameterized system of polynomial equations.  
Our main result is an algorithm that determines the cross-locus and the generators of the braid group. 
\end{abstract}

\maketitle
\section{Introduction}
Braid groups were first introduced by Emil Artin in 1925 as a generalization of the symmetric group. Consider a continuous permutation of $n$ points in the plane. The symmetric group $\mfS_n$ reflects the initial and final state of the permutation. 
In contrast, the braid group $\braidSymbol_n$ keeps track of the continuous paths of the points, as they permute. There is a natural surjective map from $\braidSymbol_n$ to $\mfS_n$. When $n\geq 2$, $\braidSymbol_n$ is an infinite group and when $n\geq 3$, $B_n$ contains a subgroup, which is free with two generators. 

In the past two decades, braid groups have been applied
%appeared 
in nuclear physics \cite{NW1996}, 
%http://www.sciencedirect.com/science/article/pii/0550321396004300
%
robotics \cite{Kur2012}, 
%See the references in this book on page 138, Algorithmic Foundations of Robotics VI:
%https://books.google.co.kr/books?id=XD-CjBnvIikC&pg=PA138&lpg=PA138&dq=Computing+braid+groups+of+graphs+with+applications+to+robot+motion+planning&source=bl&ots=cjxLKAz-L0&sig=SZqEwboNw-njsc17jWAf22g4jv4&hl=en&sa=X&ved=0ahUKEwiW1qCJ7d_MAhWEPKYKHXluDfQQ6AEISjAG#v=onepage&q=Computing%20braid%20groups%20of%20graphs%20with%20applications%20to%20robot%20motion%20planning&f=false
%
cryptography \cite{LLH2001,KLC2000,Garber2010,Dehornoy2004,CKLS2001}.
%Lee, Eonkyung; Lee, Sang Jin; Hahn, Sang Geun, Pseudorandomness from braid groups, MathSciNet ref# MR1931440 (2003g:94038)
%Ko, Ki Hyoung; Lee, Sang Jin; Cheon, Jung Hee; Han, Jae Woo; Kang, Ju-sung; Park, Choonsik, New public-key cryptosystem using braid groups, MathSciNet ref# MR1850042 (2002i:94057)
%Garber, David, Braid Group Cryptography, MathSciNet ref# MR2605310 (2011c:94044)
%Dehornoy, Patrick, Braid-Based Cryptography, MathSciNet ref# MR2105432 (2005g:94073)
%Cha, Jae Choon; Ko, Ki Hyoung; Lee, Sang Jin; Han, Jae Woo; Cheon, Jung Hee, An efficient implementation of braid groups, MathSciNet ref# MR1934520 (2003h:20066)
%
%and protein folding \cite{}. 
%This is an undergraduate paper: http://www.math.uchicago.edu/~may/VIGRE/VIGRE2011/REUPapers/Hoberg.pdf
%
%
%%Other references;
%http://www.ams.org/journals/jams/2006-19-02/S0894-0347-05-00507-2/S0894-0347-05-00507-2.pdf
% https://www.math.upenn.edu/~ghrist/preprints/singaporetutorial.pdf
In this article, we will consider braid groups of algebraic varieties defined by systems of polynomial equations.
In this setting, we have braid groups that are finitely generated subgroups of the braid group on $n$  strands.

The first case we consider is an algebraic curve $\mathcal{C}$ in $\mathbb{C}_{z}^{1}\times\mathbb{C}_{t}^{1}$ 
that maps dominantly to $\mathbb{C}_{t}^{1}$  and a generic dehomogenization of the curve $\mathbb{\hat C}$ in 
$\mathbb{P}_{z}^{1}\times\mathbb{P}_{t}^{1}$.
%(after generic homogenization {\color{blue}of $\mathbb{C}_{z}^{1}$} {\color{red}Do you mean dehomogenization here?}).
The fiber over a point in $\mathbb{C}_{t}^{1}$, consists of finitely
many $z$-coordinates and is constant over the complement of the branch locus denoted by $\cB$.
The second case we consider is a line arrange arrangement and third case is for parameterized polynomial systems.

The Galois group can
be determined using numerical homotopy continuation: for each point
in the branch locus, the fiber over a general point is tracked over
a loop encircling that branch point  and recording the associated permutation
yields a generator of the Galois group. Doing this for each branch point produces a set of generators for the the Galois
group, which is a subgroup of the symmetric group on the number of elements
in the general fiber. 
This method has been used in \cite{Po07,GP11} for curves and for higher dimensional varieties in \cite{LS09,MNMH,HRS17}

Here, we will compute the braid group $\braidSymbol$ which has more refined information than the Galois group.
Given a general loop $\gamma$ encircling a branch point, the restriction of $\mathcal{C}$ in $\mathbb{C}_Z^1\times \gamma$ becomes a set of braids in a real 3-dimensional space. 
The braid group not only keeps track of the monodromy of solutions, it also keeps track of the configuration of the braids. 

In this article, our main results are numerical algorithms 
 to compute the braid group of a curve. Our algorithms have three main steps. 
The first step is to parse the branch locus of the
curve; 
the second step computes, for each branch point, the cross-locus of a loop; 
and the third step braid-tracks around each loop. 
If each cross-locus of the fiber is proper, then our algorithm outputs the braid group. {We also give a more effective algorithm when the curve is a union of lines. Since the braid group determines other topological invariants of the compliment of the curve (e.g. the Alexander polynomial by \cite{Lib}). We hope our algorithm can also be used to study the topology of algebraic curves and their complements. }%else a regularization step is needed. 

Our paper is structured as follows. In Section \ref{sec:prelim} we recall the abstract braid group, homotopy continuation, and the branch locus. 
In Section \ref{sec:twist} we define the cross-locus over a line segment and use it to develop an algorithm to compute the braid group.
In Section \ref{sec:illustrate} we have illustrating examples and 
in Section \ref{sec:hyper} we consider the braid groups of line arrangements. 
The monodromy of these arrangements is trivial but the braid group is very interesting as seen in Example \ref{ex:hyper8} and \ref{ex:hyper46}. 
In the latter example we compute a generating set consisting of $46$ elements which have a total of 2568 crossings. 
In Section \ref{sec:poly} we discuss how the methods in the previous sections generalize to higher dimensional varieties. 

%\section{Braid groups of polynomial systems}\label{sec:poly}

% Our final section is of illustrative examples. 
Now, we introduce an example. % {\color{red}Does this belong to Section 1? Or Section 2 should be a subsection of Section 1?}

\begin{example}\label{ex:intro}%[The numbers in this example need to be recomputed--they should agree with the examples in a later section]
Consider the irreducible curve ${\mathcal C}$ in $\mathbb{C}_{z}^{1}\times\mathbb{C}_{t}^{1}$
defined by $z^{3}-t^{2}=0$ 
The projection of this curve to $\mathbb{C}_{t}^{1}$
is three to one with a branch point at $t=0$. 
Consider the loop consisting of three line segments $\gamma:\delta_{1}\to\delta_{2}\to\delta_{3}\to\delta_{1}$
where $\delta_{j}=\exp\left(\frac{2\pi \sqrt{-1}}{3}j+.2\sqrt{-1}\right)$. 

The fiber over $\delta_{1}$ is $\left(z^{(1)},z^{(2)},z^{(3)}\right)$.
The loop $\gamma$ induces a braid by keeping track of how the real and imaginary
parts of the fiber $\left(z^{(1)}\left(\gamma\right),z^{(2)}\left(\gamma\right),z^{(3)}\left(\gamma\right)\right)$
interact as the fiber is tracked along $\gamma$. 

When the fiber has distinct real parts we order $\left(z^{(1)},z^{(2)},z^{(3)}\right)$
lexicographically. 
However, there are finitely many points in $\gamma$ whose
fiber does not have distinct real parts, i.e. there is a crossing. 
 How the fibers cross will determine the twist generated by the loop $\gamma$. 
The $t$ coordinates of the crossings are: 
$(-.527\sqrt{-1},
-.510	,
-.667\sqrt{-1},
.755)$.
As we track the fiber along the loop, whenever a fiber has two points with the same real parts, we will
swap there ordering.
 Depending on the imaginary component we give
this swap a sign and call it a crossing. 
More specifically, if $z^{(i)}$
and $z^{(i+1)}$ are swapped, then we give a crossing ${\sigma_{i}}$
when $\text{imag }z^{(i)}<\text{imag }z^{(i+1)}$, and we give a crossing ${\sigma_i^{-1}}$
when $\text{imag }z^{(i)}>\text{imag }z^{(i+1)}$. 
%In the table below, we record the fibers with a crossing along a loop. 
Our convention is to
order the points in the fiber with a crossing by the induced ordering of the incoming paths.
We see that the fiber over the $t$-coordinates 
$(-.527\sqrt{-1},
-.510	,
-.667\sqrt{-1},
.755)$ 
have crossings 
$({\sigma_{2}},
 {\sigma_{1}},
 {\sigma_{2}},
 {\sigma_{1}})$ 
 respectively.
Together, these four crossings make a sequence that encodes a twist ${\bf b}=
\left(
 {\sigma_{2}}
 {\sigma_{1}}
{\sigma_{2}}
 {\sigma_{1}}
\right)$.
Note that the twist acts on the fiber to induce a permutation
$\left[1,2,3\right]\to\left[2,3,1\right]$; this means the braid group
maps onto the monodromy group. 

\[
\arraycolsep=1.4pt\def\arraystretch{1.65}
\hspace{-8em}
\begin{array}{c|cccccc}
t & 
& 	\pi_{-1}(t)&  \\	%&\operatorname{Twists}\\
\hline
.527\sqrt{-1} 	&
-.652   &     .326-.56\sqrt{-1}   &     .326+.565\sqrt{-1}\\%&\sigma_2\\
-.510		&
-.319-.552\sqrt{-1}   &     -.319+.552\sqrt{-1}   &     .638		
\\%&\sigma_1\\
-.667\sqrt{-1}		&
-.763   &     .381-.661\sqrt{-1}   &     .381+.661\sqrt{-1}		
\\%&\sigma_2\\
.755		&
-.414-.718\sqrt{-1}   &     -.414+.718\sqrt{-1}   &     .829		
%&\sigma_1
\end{array}
 \put(5,0){$
 \begin{array}{c}
 \text{Braid}\\
 \hline
\begin{tikzpicture}
\braid[number of strands=3,height=.25in] (braid) a_{2} a_{1} a_2 a_1 ;
\end{tikzpicture}
\end{array}$}
\]

\end{example}

\section{Preliminaries }\label{sec:prelim}

In this section we set notation for braid groups and recall homotopy
continuation.

\subsection{Abstract braid groups}
In this subsection we recall definitions of braid groups. 

The \emph{abstract braid group on $n$ strands}, denoted $\bf{B}_n$, has the following presentation: 
$$
\begin{array}{lll}
{\bf B}_n&:=
&\big\langle 
\sigma_1,\sigma_2,\dots,\sigma_{n-1} 
\quad
\text{ such that},
 \quad 
\sigma_{i}\sigma_{i+1}\sigma_i=\sigma_{i+1}\sigma_{i} \sigma_{i+1}, 
\quad
 \text{and }\quad
\sigma_{i}\sigma_{j}=\sigma_{j}\sigma_{i}
\big\rangle,
\end{array}
$$
where  $1\leq i\leq n-2$ in the first group of relations and  $|i-j|\geq 2$ in the second group of relations. 
%For convenience we denote the inverse of $\sigma_{i}$ as $\sigma_i^{-1}$.
We will consider subgroups of ${\bf B}_n$, by specifying a set of generators. 

\begin{example}\label{b3}
Consider the case where $n=3$. 
The following two elements generate ${\bf B}_3$ and to the right are their inverses.

$$
\begin{array}{cccccc}
\begin{array}{cc}
\sigma_1 \\
\hline
\centering{
\begin{tikzpicture}
\braid[number of strands=3,height=1in] (braid) a_{1} ;
\end{tikzpicture}
}
\end{array}
&
\begin{array}{cc}
\sigma_2 \\
\hline
\centering{
\begin{tikzpicture}
\braid[number of strands=3,height=1in] (braid) a_{2} ;
\end{tikzpicture}
}
\end{array}
&
%\hline
%
\begin{array}{cc}
\sigma_1^{-1}\\
\hline
\centering{
\begin{tikzpicture}
\braid[number of strands=3,height=1in] (braid) a_{1}^{-1} ;
\end{tikzpicture}
}
\end{array}
&
\begin{array}{cc}
\sigma_2^{-1}\\
\hline
\centering{
\begin{tikzpicture}
\braid[number of strands=3,height=1in] (braid) a_{2}^{-1} ;
\end{tikzpicture}
}
\end{array}\\
\end{array}
$$
The braid group in Example \ref{ex:intro}
is generated by $\sigma_2\sigma_1\sigma_2\sigma_1$.
\end{example}

\subsection{Homotopy continuation and branch points}\label{ss:branchPoints}
In this subsection we describe homotopy continuation and branch points.

For us, a curve ${\cC}\subset\mathbb{C}_{z}^{1}\times\mathbb{C}_{t}^{1}$
will be defined by a polynomial $f\left(z,t\right)=0$ and we assume that the projection to 
$\mathbb{C}_t^1$ is dense.

\begin{definition}
The \emph{branch locus} $\branch$ is defined to be the following:
\begin{equation}\left\{
t : \text{there exists } z\text{ such that } f(z,t)=0,\frac{\partial f}{\partial z}=0
\right\}.
\end{equation}
%%
%\begin{equation}
%\left\{ t : \frac{\partial^kf}{\partial z}=0\text{ where }k=\deg_zf\right\}\end{equation}
%{\color{blue}which is in fact the union of the standard branch locus and the locus of the points at infinity.} {\color{red}If the curve $\mathcal{C}$ is after a general dehomogenization, then the points at infinity will not introduce any non-trivial element in the braid group. }
\end{definition}
%{\color{red}Does a general dehomogenization get rid of all points of type (2)?}

Let $\delta',\delta''$ denote distinct complex numbers in $\mathbb{C}_t^1$ and denote 
the \emph{line segment from $\delta'$ to $\delta''$}  by 
 ${\Ldelta}\subset\mathbb{C}_t^1$. 
Suppose that $\Ldelta$ does not intersect the branch locus $\branch$.
Then, the fiber of the curve $\mathcal{C}$ defined by $f(z,t)=0$ over $\Ldelta$ consists of 
$\deg\pi:\mathcal{C}\to\mathbb{C}_t^1$~paths. 
Using predictor-corrector methods we are able to track the start points of the paths to the end points of the paths using numerical methods \cite{li_1997,allgower2012numerical,allgower1993continuation}. 
Given the fiber over a point $p$ in $\Ldelta$, these methods use Euler's method to approximate the fiber over a point near $p$ in  $\Ldelta$ and Newton's method to refine the solutions. 
Together, this is called homotopy continuation. 
There exist off the shelf software \cite{Bertini,PHC,bertini4M2,BMonodromy} for these methods that use reliable heuristics to track the paths quickly. 
%Moreover, there is a theory on how to certify that these paths are tracked correctly by keeping track of the numerical errors and bounding the derivatives \cite{}. 

\subsection{Braid-tracking}\label{ss:braidTracking}

Suppose $\Path$ is a path in $\mathbb{C}_t^1$.
With standard homotopy continuation, we are interested in understanding the behavior of the fibers over $\Path$ as paths in a three-dimensional space. % the behavior of the endpoints of paths of points of the fiber over $\Path$. 
In \emph{braid-tracking}, we will be interested in how the paths 
intertwine with one another by considering them in an ambient $\mathbb{R}^2\times P$ rather than $\mathbb{C}^1_z\times \Path$.

To do braid-tracking along $\Path$, we assume that points in the path $\Path\subset\mathbb{C}_t^1$ 
have fibers with $n$ distinct points, i.e., $\Path\subset\mathbb{C}_t^1\setminus\branch$.
By identifying $\mathbb{C}_z^1$ with the canonical $\mathbb{R}^1_x\times\mathbb{R}^1_y$ by real and imaginary parts we will keep track of how the paths intertwine.
The points in the fiber are ordered lexicographically (by $x$-coordinates) from most negative real parts to most positive real parts. This is a well defined ordering when 
each point of the fiber has distinct $x$-coordinates.
A point $\rho\in\mathbb{C}_{t}^{1}\backslash{\cB}$ is said to
have a $\emph{crossing}$ whenever the real parts of $z^{(i)}$
in the fiber $\pi^{-1}\left(\rho\right)=\left\{ z^{\left(1\right)},z^{\left(2\right)},\dots,z^{\left(n\right)}\right\} $
are not distinct. 
We say that that the fiber is \emph{crossed}.

Let \emph{$\Cross_\pi(P)$} denote the subset of points in $\Path$ with crossings,
 and we say $\Cross_\pi(P)$ is the \emph{cross-locus of $\Path$.} 
We call a the fiber over a point in $\Cross_\pi(P)$ 
  \textit{proper}, 
  if given any real number $x$, at most two points in $\left\{ z^{\left(1\right)},z^{\left(2\right)},\dots,z^{\left(n\right)}\right\}$ have real part equal to $x$.

 We call the fiber over a point $p$ in $\Cross_\pi(p)$  \textit{transversal}, if it is proper and if the graphs of $\left\{ z^{\left(1\right)},z^{\left(2\right)},\dots,z^{\left(n\right)}\right\}$ as functions of $s$ intersect transversally at this fiber. 
In this case,  a well ordering of the fiber is induced by the path coming from the left. 
We denote this ordered fiber by $\piLeft(t)$.
We defined $\piRight(t)$ to be the the result of acting on $\piLeft(t)$ by the crossings.
% {\color{red} Need to be more precise with the definition of swaps.} 

In a transversal crossed fiber, an adjacent pair or multiple adjacent pairs of
elements can have the same real part. 
If the real parts of $z^{\left(i\right)}$
and $z^{\left(i+1\right)}$ are equal, then we return the crossing $\sigma_{i}$
when $\text{imag }z^{\left(i\right)}<\text{imag }z^{\left(i+1\right)}$
and we return the crossing ${\sigma_{i}^{-1}}$ when $\text{imag }z^{\left(i\right)}>\text{imag }z^{\left(i+1\right)}$. 

We are interested in understanding the behavior of crossed fibers
over points in a loop $\gamma$. 
We say a loop \emph{encircles}
the branch point $\tau$ in $\mathbb{C}_t^1$ if winding number of the loop about $\tau$ is one and the interior of the loop does not contain any points in $\cB\setminus\{\tau\}$.

\begin{theorem}
%{\color{red}We need a base point for this theorem just like in Figure 5, and ${\bf b}_\tau$ should consist of a path from the base point to the loop, the loop, and the reverse path.}
Let ${\cB}$ be the branch locus for the curve ${\cC}$, and
let $\beta\in\mathbb{C}_t^1\setminus\cB$ denote a base point. 
For $\tau\in{\cB}$, let $\gamma$ be a loop from $\beta$ encircling $\tau$.
Furthermore, let ${\bf b}_\tau$ be the twist associated to the loop
$\gamma$. 
Then, the braid group of ${\cC}$ is generated by
$\left\{ {\bf b}_{\tau}\mid\tau\in{\cB}\right\} $ as a subgroup
of the abstract braid group $\braidSymbol_{k}$, where $k$ is the
degree of the projection ${\cC}\to\mathbb{C}_{t}^{1}$. 
\end{theorem}
\begin{proof}
The fundamental group of $\mathbb{C}_t^1\setminus\branch$ is generated by the loops around each branch point. Any loop in $\mathbb{C}_t^1\setminus\branch$ is homotopy equivalent to a path represented by an element in $\pi_1(\mathbb{C}_t^1\setminus\branch, \text{base point})$.
\end{proof}

\section{Cross-locus of a line segment}\label{sec:twist}

In this section we develop a set of equations that allow us to determine the cross-locus of a line segment. 
We then determine the cross-locus of the path $\gamma$ by considering it as concatenation of finitely many line segments.
To a loop $\gamma$ (with base point $\beta$), we will associate a sequence of crossings to make an element of the braid group ${\bf b}_{\gamma}$. 
Let ${\cT}_{\gamma}$ denote the sequence of points in $\gamma$ that are ordered by their relative position on the path $\gamma$ and suppose $\cT_\gamma$ contains $\Cross_\pi(P)$.
 We call ${\cT}_{\gamma}$ an {\emph{ordered superset of the cross-locus of the loop $\gamma$} (for the projection of the curve ${\cC}$). 
The braid associated to the loop $\gamma$, denoted ${\bf b}_{\gamma}$, is the sequence of crossings of fibers induced by the respective sequence of points in ${\cT}_{\gamma}$.

%Tracking multiple paths at the same time can be thought of as working on the fiber product. 
%Braid tracking keeps track of the order of the solutions in a a prescribed order.
%To braid track over a segment the end points need to be untwisted and the segment must not contain any branch points. 
%
%The end points over the untwisted fiber are ordered by real parts from most negative to most positive. 
%We say we have good braid tracking if the real parts come together in pairs (triples and higher are not allowed)

%\section{Twist loci and branch loci}
%In the previous section we described how to do smooth braid tracking when the twist locus is finite. 
%In this section, we give conditions when the twist locus is finite. 
%Furthermore, we prescribe finitely many loops to braid-track over to generate the braid group of the curve.

\subsection{Finiteness of the twist locus}
One important fact is the following. 
%{\color{red} The format of the Lemma is strange. And if I add any words before the``begin enumerate", latex gives me error. }

\begin{lemma}\label{lem:generalSegment}
Let $\mathcal{C}$ be the curve in $\mathbb{C}_z\times \mathbb{C}_t$ defined by $f(z, t)=0$. 
Let $\lambda$ be a nonzero complex number. 
Denote the curve in $\mathbb{C}_z\times \mathbb{C}_t$ defined by $f(\lambda z, t)=0$ by $\mathcal{C}_\lambda$. 
Suppose $\lambda$ is general. Then for curve $\mathcal{C}_\lambda$, the following statements hold. 

\begin{enumerate}
\item\label{i1} A general loop $\gamma$ in $\mathbb{C}_{t}^{1}$ has finitely many points
with a crossed fiber. 
\item\label{i2} If a general loop has only proper crossed fibers, then it has only transversal crossed fibers.
\item\label{i3} If the curve $C$ is irreducible, then a general loop has only proper (hence transversal) crossed fibers. 
\end{enumerate}

\end{lemma}
\begin{proof}
Let $S_{\lambda}\subset \mathbb{C}_t$ be the subset where $C_\lambda$ has a crossed fiber. Let $S_B\subset \mathbb{C}_t$ be the subset where at least three points of the fibers of $C$ are on a real line in $\mathbb{C}_z$. 

Clearly, $S_B$ and $S_\lambda$ are real algebraic subsets of $\mathbb{C}_t$. The set $\mathcal{S}=\{(z, t, \lambda)\in \mathbb{C}_z\times \mathbb{C}_t\times \mathbb{C}^*_\lambda|z\in S_\lambda\}$ is a real algebraic subset of $\mathbb{C}_z\times \mathbb{C}_t\times \mathbb{C}^*_\lambda$ of real codimension one. Therefore, for a general $\lambda$, $S_\lambda$ is of real codimension one. Thus for such $\lambda$, a general loop only intersects $S_\lambda$ at finitely many points. Hence (\ref{i1}) follows. 

Given (\ref{i1}), we only need to prove (\ref{i2}) locally in $\mathbb{C}_t$. Locally, we can express the fibers $z^{(1)}, \ldots, z^{(k)}$ as holomorphic functions of $t$. Notice that $Re(z^{(i_1)})-Re(z^{(i_2)})=Re(z^{(i_1)}-z^{(i_2)})$ is the real part of a holomorphic function. According to Cauchy-Riemann equation, as functions of $t$, the critical points of $Re(z^{(i_1)}-z^{(i_2)})$ is equal to the critical points of $z^{(i_1)}-z^{(i_2)}$. Hence the critical points of $Re(z^{(i_1)}-z^{(i_2)})$ are isolated. Thus, the real locus defined by the equation $Re(z^{(i_1}))=Re(z^{(i_2)})$ is generically reduced for $i_1\neq i_2$. Therefore, a general loop will intersect the locus $Re(z^{(i_1)})=Re(z^{(i_2)})$ transversally, and hence a general loop has only transversal crossed fibers. 

For part (\ref{i3}), we only need to show that for a general $\lambda$, the intersection $S_B\cap S_\lambda$ is empty or zero-dimensional, because every loop that does not passes through $S_B\cap S_\lambda$ only have proper crossed fibers. So we only need to consider the cases when $S_B$ is of real dimension one or two. 

Suppose $C$ is irreducible. We claim that $S_B$ is a proper subset of $\mathbb{C}_t$. Suppose $S_B=\mathbb{C}_t$. As before, we can consider $z^{(1)}, \ldots, z^{(k)}$ as local holomorphic functions of $t$. Then locally, there exist $i_1, i_2$ and $i_3$ such that the ratio $(z^{(i_3)}-z^{(i_1)})/(z^{(i_2)}-z^{(i_1)})$ locally takes values in real numbers. However, $(z^{(i_3)}-z^{(i_1)})/(z^{(i_2)}-z^{(i_1)})$ is holomorphic. Even locally, the only real valued holomorphic functions are constant functions. Thus the ratio $(z^{(i_3)}-z^{(i_1)})/(z^{(i_2)}-z^{(i_1)})$ is a constant real number. Thus, this ratio is preserved by any monodromy translation. Without loss of generality, we can assume that locally $z^{(i_2)}$ is on the line segment between $z^{(i_1)}$ and $z^{(i_3)}$. Since $C$ is irreducible, the Galois action on the fiber is transitive. Hence there exists a Galois action that maps $z^{(i_2)}$ to another fiber, which lies on a vertex of the convex hall spanned by the points in the fiber. However, the Galois action can be realized by a monodromy translation through a loop in $\mathbb{C}_t$. Throughout this loop, the translation of $z^{(i_2)}$ always lie in between two other points in the fibers. This is a contradiction that the image of $z^{(i_2)}$ under the Galois action becomes a vertex on the convex hall of all fibers. 

Now, the only case left is when $S_B$ has real dimension 1. Suppose $S_B$ has $l$ irreducible 1-dimensional components. Let $\lambda_0, \lambda_1, \ldots, \lambda_l$ be general numbers in $\mathbb{C}^*$. Clearly, any pair of $S_{\lambda_0}, S_{\lambda_1}, \ldots, S_{\lambda_l}$ does not share a 1-dimensional irreducible component. Therefore, for some $i$, $S_{\lambda_i}$ does not share a 1-dimensional irreducible component with $S_B$, and hence $S_{\lambda_i}\cap S_B$ is of dimension zero. This construction shows that for a general $\lambda$, $S_{\lambda}\cap S_B$ is of dimension zero.

%For part (\ref{i2}), we fix a general $\lambda$. Suppose a general loop has only proper crossed fibers. 
\end{proof}

\begin{remark}\label{badcrossing}
When the curve $C$ is reducible, it is possible that some crossed fiber of $C_\lambda$ on a general loop contains more than two points. For example, suppose $C$ is defined by $(z-t)(z-2t)(z-3t)=0$ or $z(z^2-t)=0$. Then every crossed fiber has three points. If a crossed fiber has $l$ elements, then the crossing at this point can be represented by a product of $(l-1)!$ standard generators and their inverses. 
We do not pursue this direction here. 
%We will not pursue this direction in this present paper. {\color{red} Please check whether the last sentence sounds right.}
\end{remark}

\subsection{Equations for the cross-locus}\label{ss:mainResult}

In this subsection we describe how to compute a finite set $\cT_\gamma$ containing
the cross-locus.
This is done by using equations \eqref{eq:twistLocus} to determine such a set over a line segment  
$\Ldelta$ for the projection 
$\pi:\mathcal{C}\to\mathbb{C}_t^1$.

\begin{prop}\label{prop:eq}
Let $s$ in the interval $[0,1]$ parameterize $\Ldelta.$
The isolated solutions to the following set of equations with $s\in[0,1]$ contain the cross-locus, 
where for $a\in \mathbb{C}$, 
$\bar a$ denotes the complex conjugate of $a$,  
$f(z,t)=\sum_{i,j}a_{i,j}z^it^j$, and $g(z,t):=\sum_{i,j}\bar a_{i,j}z^it^j$.

\begin{equation}\label{eq:twistLocus}
\begin{aligned}
f\left(x+\sqrt{-1}y_1,(1-s)\delta'+s\delta''\right)=0,    
&&  g\left(x-\sqrt{-1}y_1,(1-s)\bar\delta'+s\bar\delta''\right)=0\\
 f\left(x+\sqrt{-1}y_2,(1-s)\delta'+s\delta''\right)=0,   
 &&	g\left(x-\sqrt{-1}y_2,(1-s)\bar\delta'+s\bar\delta''\right)=0
\end{aligned}
\end{equation}

\end{prop}

\begin{proof}
%{\color{red}Every twist locus will appear as the solution of the above equations. Conversely, do we know that any solution of the above equation (with distinct $y_1, y_2$) gives a twist locus?}
The first equation replaces $z$ with $x+\sqrt{-1}y$. 
The second equation is the complex conjugate of the first so that the real solutions $(x,y_1,s)$ coincide with $(z=x+\sqrt{-1}y_1,s)$. 
The last pair of equations yield the fiber product over $(x,s)$ with respect to $y_1,y_2$. 
The isolated solutions are outside the locus where $y_1=y_2.$ 
\end{proof}

\subsection{Computing the braid group (main result)}

Algorithm \ref{algo:braid} produces a generator of the braid group. 

\begin{algorithm}
\DontPrintSemicolon % Some LaTeX compilers require you to use \dontprintsemicolon instead
\KwIn{
\begin{enumerate}
\item $f\left(z,t\right)=0$ defining a curve. 
\item A finite set $R$ containing the branch locus ${\cB}$ of our curve. 
%\item The endpoints $\delta_{i}$ of a loop $\gamma$ based at $\beta$ given by straight line segments encircling precisely the branch point of ${\cB}$ above, i.e. $\gamma:\left(\delta_{0}\to\delta_{1}\to\dots\to\delta_{4}\to\delta_{5}\right)$ where $\beta=\delta_{5}=\delta_{0}$. 
\end{enumerate}
}
\KwOut{A set of generators of the braid group of the curve $\mathcal{C}.$}
Set $T$ to be the empty set.\;
Set $\bf{B}$ to be the empty set.\;
\For{  $\tau$ in $R$ } {\label{item:important} 
  Construct a loop $\gamma$ encircling the branch point $\tau$ such that it is a concatenation of line segments $G$. \;
  \For{ $\Ldelta$ in $G$}{
  Determine a finite set $D$ containing the cross-locus of $\Ldelta$ by solving \eqref{eq:twistLocus}.\;
  Order the elements of $D$ with respect to the path and append to $T$.}
Set $\bf{b}$ to be the empty set (This corresponds to the identity element of the braid group).\;  
\For{ $p$ in $T$}{
  Use homotopy continuation to track and determine the crossing. \;
  Append the crossing to $\bf{b}$.\;
  \If{$\bf{b}\neq\emptyset$} {
  Append the twist induced by taking the product of crossings of $\bf{b}$ to $\bf{B}$.}    }}
\Return{{\bf B} }\;
\caption{Generators of the braid group}
\label{algo:braid}
\end{algorithm}

\begin{proof}[Proof of Correctness]
The proof of correctness follows from Lemma \ref{lem:generalSegment} and Proposition \ref{prop:eq}.
\end{proof}

\section{Illustrative examples}\label{sec:illustrate}
In this section we consider four examples.

%%EXAMPLE 1%%%%%%%%%%%%%%%%%%%%%%%%%%%%%%%%%%%%%%%%%%%%%%%%%%%%%%%%%%%%%%%%%%%%%%%%%%%%%%
\begin{example}\label{ex:oneBranchpoint}
Let $\mathcal{C}_i$ be defined by 
$f_i:=z^{3}-t^{i}=0$ for $i=1,2$. 
%The monodromy groups of the projections of these curves to $\mathbb{C}_t^1$ are indistinguishable.
Each of these curves have a unique branch point at $t=0$. 
Using Algorithm \ref{algo:braid} we compute the generator of each braid group of $\pi:\mathcal{C}_i\to\mathbb{C}_t^1$ as seen in Figure~\ref{fig:oneBranchPoint}.

\begin{figure}[hbt!]
\hspace{-16em}$
\begin{array}{c}%%Two rows
%%%%Row 1
\begin{array}{cc}
%Picture
{\includegraphics[scale=0.50]{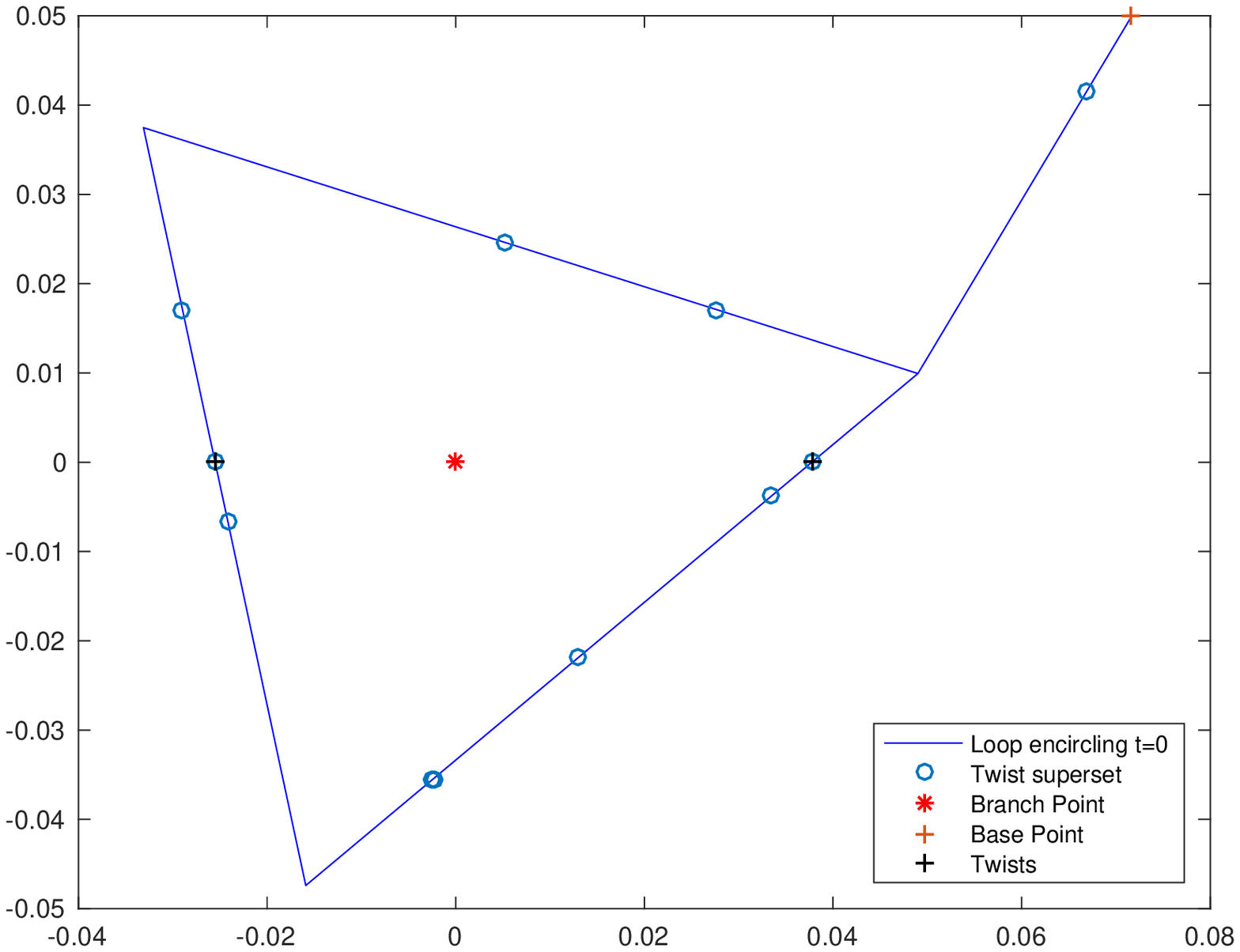}}&
%Braid
\put(0,120){
$
\begin{array}{c}
\text{Braid for $z^3-t$}\\
\hline
\begin{tikzpicture} 
\braid[number of strands=5,height=.5in,
       style strands={4,5,6}{draw=none}]] (braid)  a_4 a_2 a_1 a_4   ;
% \put(95,-30){\scriptsize$t=.0263 \sqrt{-1}$} 
 \put(95,-64){\scriptsize$t=-.0255$} 
 \put(95,-104){\scriptsize$t=-.0333\sqrt{-1}$} 
 %\put(95,-140){\scriptsize$t=.0377$} 
\end{tikzpicture}
\end{array}
$}
\end{array}\\
%%%Row 2
\begin{array}{cc}
%Picture
{\includegraphics[scale=0.50]{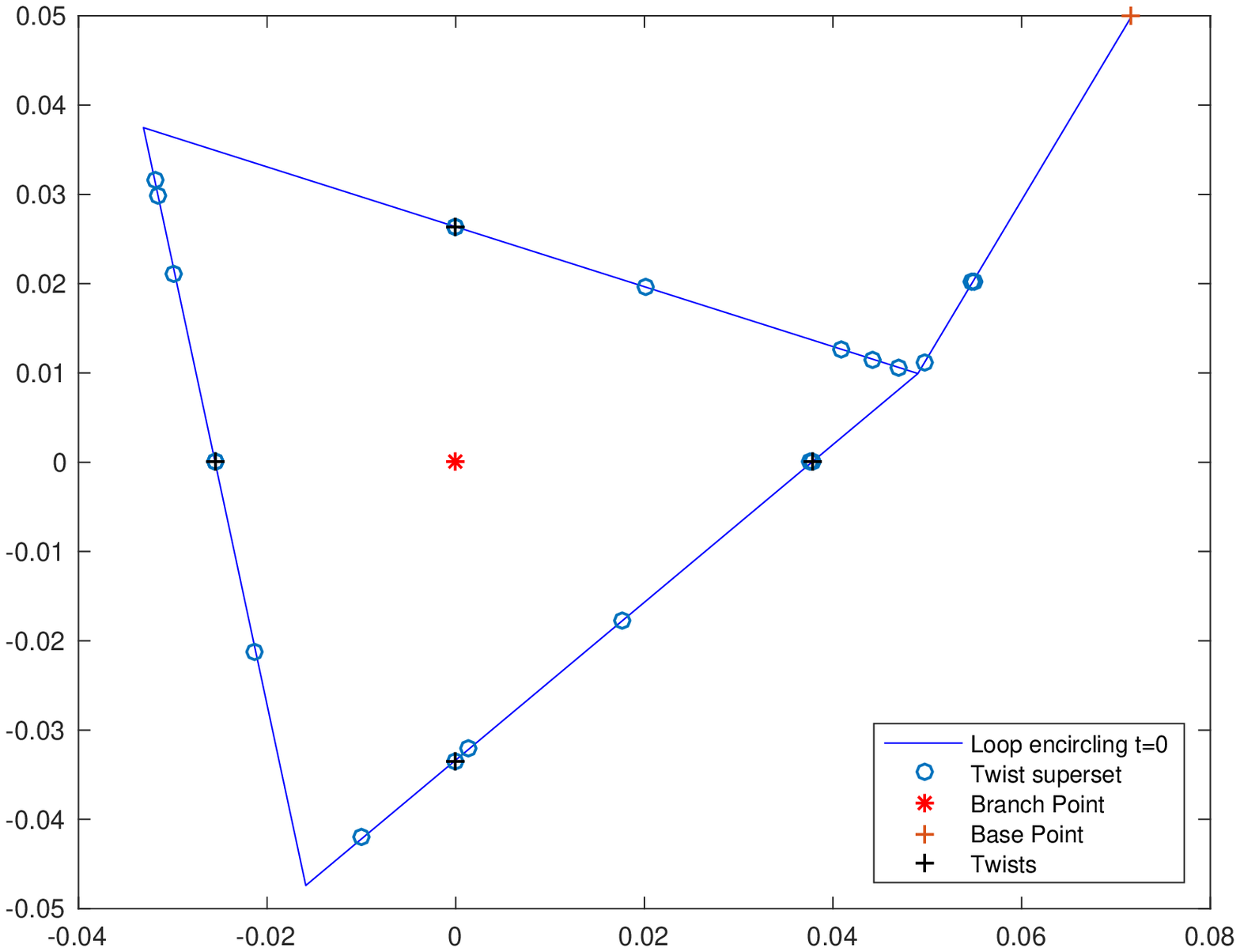}}&
%Braid
\put(0,120){
$
\begin{array}{c}
\text{Braid for $z^3-t^2$}\\
\hline
\begin{tikzpicture}
\braid[number of strands=5,height=.5in,
       style strands={4,5,6}{draw=none}]] (braid)  a_{2} a_{1} a_2 a_1   ;
 \put(95,-30){\scriptsize$t=.0263 \sqrt{-1}$} 
 \put(95,-64){\scriptsize$t=-.0255$} 
 \put(95,-104){\scriptsize$t=-.0333\sqrt{-1}$} 
 \put(95,-140){\scriptsize$t=.0377$} 
\end{tikzpicture}
\end{array}
$}
\end{array}
\end{array}
$
\caption{
The xy-axis of the plots are the real and imaginary parts of $t$.
We vary $t$ along the plotted five directed line segments when performing homotopy continuation counter clockwise along the triangle.  
%Locus $k$ is the set of solutions to $\eqref{}$ along the $k$th line segment.
 }\label{fig:oneBranchPoint}
\end{figure}

\iffalse
\begin{figure}[!] 
\includegraphics[scale=0.4]{ex1_tDegree1.eps}
\includegraphics[scale=0.4]{ex1_tDegree2.eps}
\caption{
The xy-axis of the plots are the real and imaginary parts of $t$.
We vary $t$ along the plotted $5$ directed line segments plotted when performing homotopy continuation counter clockwise along the triangle.  
Locus $k$ is the set of solutions to $\eqref{}$ along the $k$th line segment.
 }
\end{figure}
\fi
\iffalse
\begin{tabular}{|c|c|c|c|c|c|}
\hline 
 & $M_{1}$ & $M_{2}$ & $M_{3}$ & $M_{4}$ & $M_{5}$\tabularnewline
\hline 
\hline 
\#Locus for $\pi_{1}$ & 2 & 2 & 3 & 5 & 1\tabularnewline
\hline 
\#Locus for $\pi_{2}$ & 1 & 6 & 6 & 6 & 4\tabularnewline
\hline 
% $\pi_{1}$ Twists  &  &  & $\sigma_2$ at $.0255$ &  $\sigma_1$ at $.0377$  & \tabularnewline
%\hline 
 %$\pi_{2}$ Twists &  &  ${[}2,3{]}$ at $.0263\sqrt{-1}$  & {[}1,2{]} at $ -.0255$ & &\tabularnewline
 %		 &  &  &  {[}2,3{]} at $-.0333\sqrt{-1}$ &{[}1,2{]} at $.0377$ &  \tabularnewline
\hline 
\end{tabular}
\fi
\end{example}

\pagebreak
%%EXAMPLE 2%%%%%%%%%%%%%%%%%%%%%%%%%%%%%%%%%%%%%%%%%%%%%%%%%%%%%%%%%%%%%%%%%%%%%%%%%%%%%%
\pagebreak
\begin{example}\label{ex:twoBranchpoint}
Let $\mathcal{C}$ be defined by 
$f=z^4-4z^2+3+t$. 
The projection $\pi:\mathcal{C}\to\mathbb{C}_t^1$
has two branch points as seen in Figure \ref{fig:twoBranchPoints}.
We take a loop around each branch point that is a concatenation of straight line segments. 
Each loop has precisely one point in the cross-locus. 
We consider the points of the fiber of each  cross point. 
In each fiber, we swap the  order of the points accordingly
 and record a generator of the braid group. 
\begin{figure}[htb] 
\includegraphics[scale=.6]{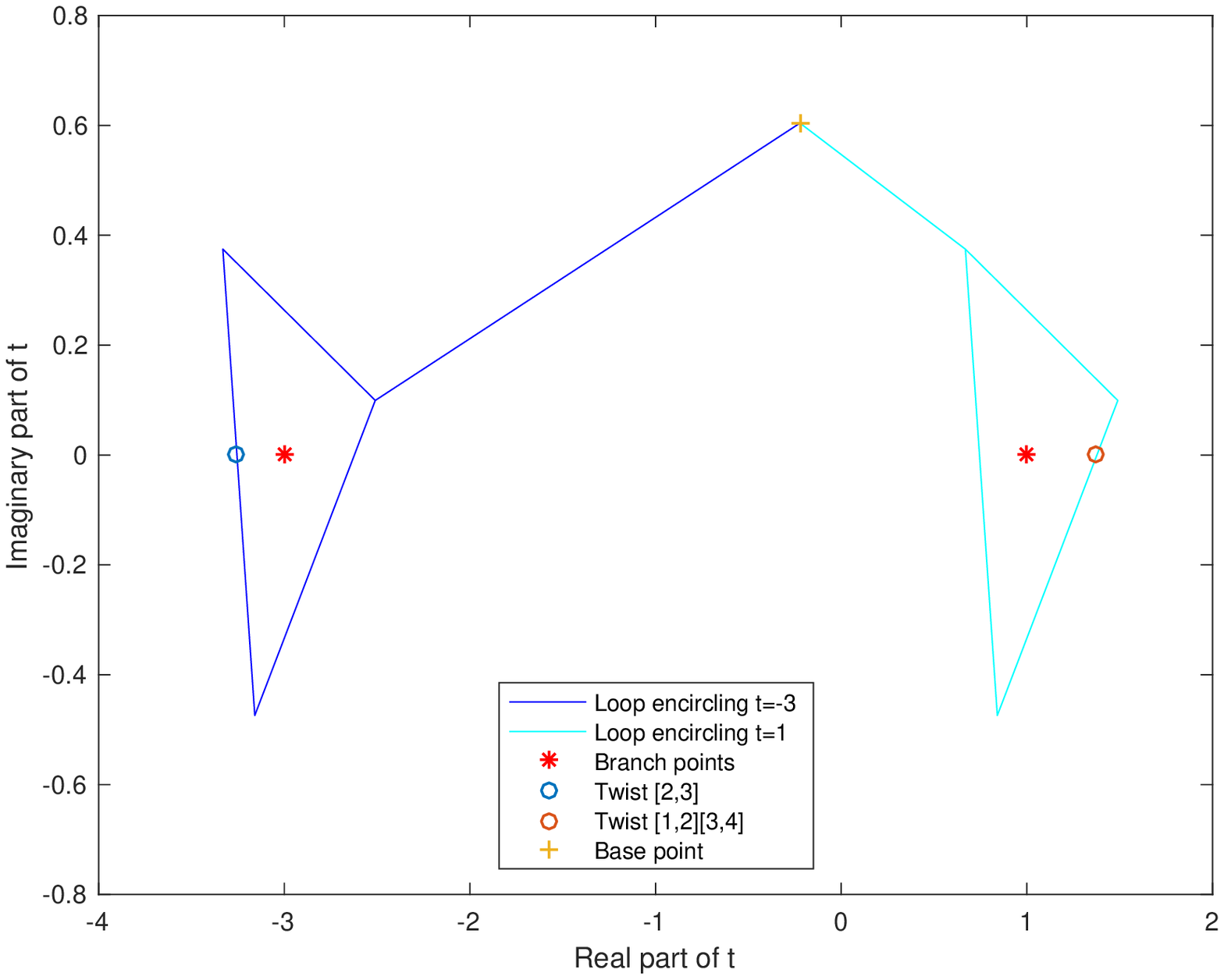}
\end{figure}

\begin{figure}[htb] 
\hspace{-20em}$
\begin{array}{cc}
\includegraphics[scale=0.5]{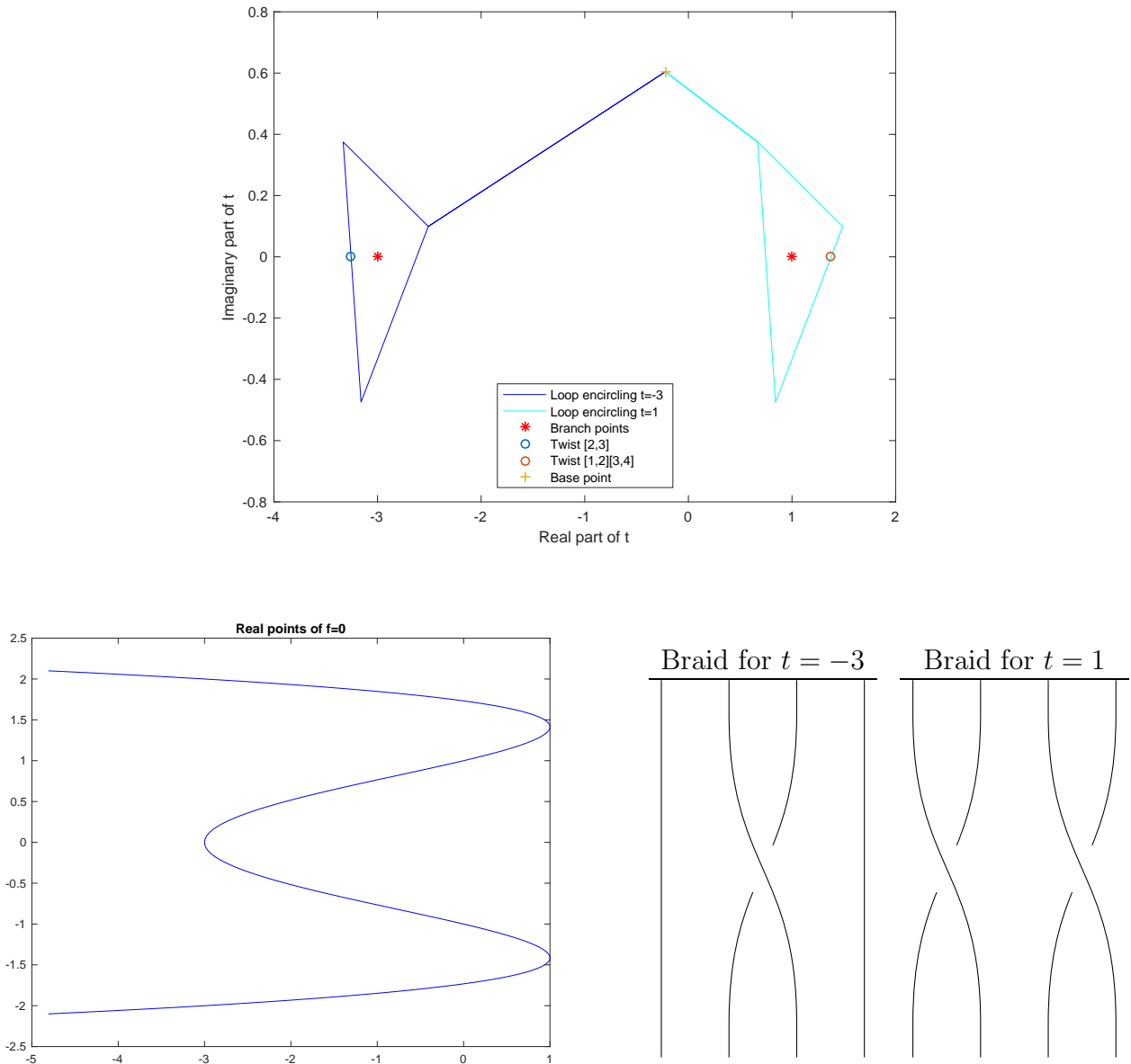}&
%%%%%%%%%%%
\put(0,100){$
\begin{array}{ccc}
\begin{array}{c}
\text{Braid for $t=-3$}\\
\hline
\begin{tikzpicture}
\braid[number of strands=4,height=2in] (braid)  a_2   ;
\end{tikzpicture}
\end{array}
&
\begin{array}{c}
\text{Braid for  $t=1$}\\
\hline
\begin{tikzpicture}
\braid[number of strands=4,height=2in] (braid)  a_{3}-a_{1}   ;
\end{tikzpicture}
\end{array}
\end{array}
$}
%%%%%%%%%%
\end{array}
$
%%How to make the pictures algined?
\caption{Real points of $\mathcal{C}$ and the associated braids.}\label{fig:twoBranchPoints}
\end{figure}

\end{example}

\pagebreak
%%EXAMPLE 3%%%%%%%%%%%%%%%%%%%%%%%%%%%%%%%%%%%%%%%%%%%%%%%%%%%%%%%%%%%%%%%%%%%%%%%%%%%%%%

\begin{example}
We consider an example related to dessin d'enfant. 
Let 
$$g_i:=z^3\left(z^2-2z+\alpha_i\right)^2
\text{ for }i=1,2
\text{ where }
\alpha_1=\frac{34+6\sqrt{21}}{7},  %alpha
\,
\alpha_2=\frac{34-6\sqrt{21}}{7}.   %beta
$$
\end{example}
The polynomials $g_1,g_2$ are Galois conjugate and have only two finite branch points.
Let $\mathcal{D}_i$ be defined by $f_i$ where 
$$
f_1(z,t):=-\frac{1}{1000}g_1-t\quad
f_2(z,t):=g_2-\frac{t}{20}.
$$
We compute the braid group for the curves $\mathcal{D}_i$. {See Figure \ref{d1}, \ref{d2} and \ref{d12}. }

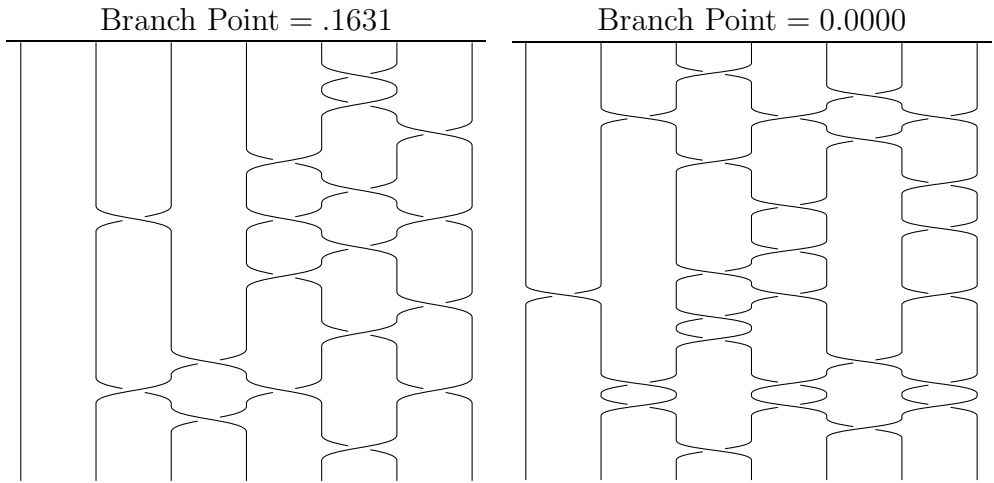
\begin{figure}
$$
\begin{array}{cc}
\begin{array}{c}
\text{Branch Point}=.1631\\
\hline
\centering{
\begin{tikzpicture}
\braid[number of strands=7,height=.15in] (braid) 
a_5
a_5^{-1}
a_6
a_4^{-1}
a_5
a_2-a_4-a_6^{-1}
a_5
a_4^{-1}
a_6
a_5^{-1}
a_3
a_2^{-1}-a_4-a_6^{-1}
a_3
a_5^{-1}
;
\end{tikzpicture}
}
\end{array}
&
\begin{array}{c}
\text{Branch Point}=0.0000\\
\hline
\centering{
\begin{tikzpicture}
\braid[number of strands=7,height=.11625in] (braid)
a_3^{-1}
a_5
a_2-a_4^{-1}-a_6
a_5
a_3^{-1}
a_6^-1
a_4
a_6
a_4^{-1}
a_3
a_1-a_4-a_6^{-1}
a_3
a_3^{-1}
a_5
a_2-a_4^{-1}-a_6
a_2^{-1}-a_4-a_6^{-1}
a_5^{-1}
a_3;
\end{tikzpicture}
} 
\end{array}
\end{array}
 $$\caption{
 This is the braid for the $\mathcal{D}_1$.
 }\label{d1}
\end{figure}

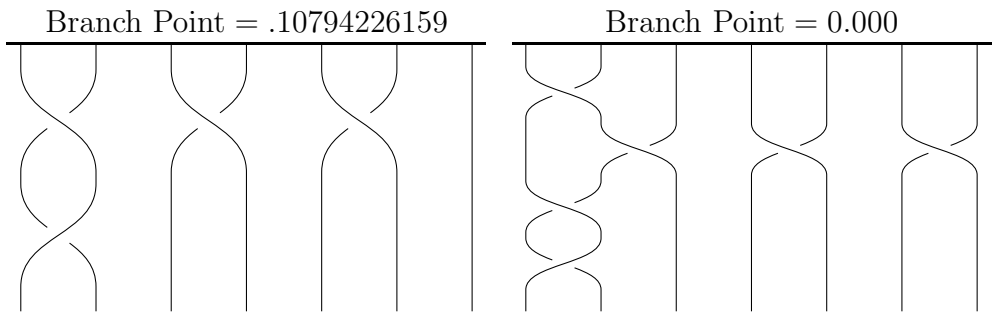
\begin{figure}
$$
\begin{array}{cc}
\begin{array}{c}
\text{Branch Point}=.10794226159\\
\hline
\centering{
\begin{tikzpicture}
\braid[number of strands=7,height=.6in] (braid)
a_1-a_3-a_5 a_1^{-1} ;
\end{tikzpicture}
} 
\end{array}
&
\begin{array}{c}
\text{Branch Point}=0.000\\
\hline
\centering{
\begin{tikzpicture}
\braid[number of strands=7,height=.3in] (braid) 
a_1 a_2-a_4-a_6 a_1 a_1^{-1};
\end{tikzpicture}
}
\end{array}
\end{array}
 $$
\caption{These are the two generators of the braid group for $\mathcal{D}_2.$}\label{d2}
\end{figure}

\begin{figure}[hbt!] 
\includegraphics[scale=0.37]{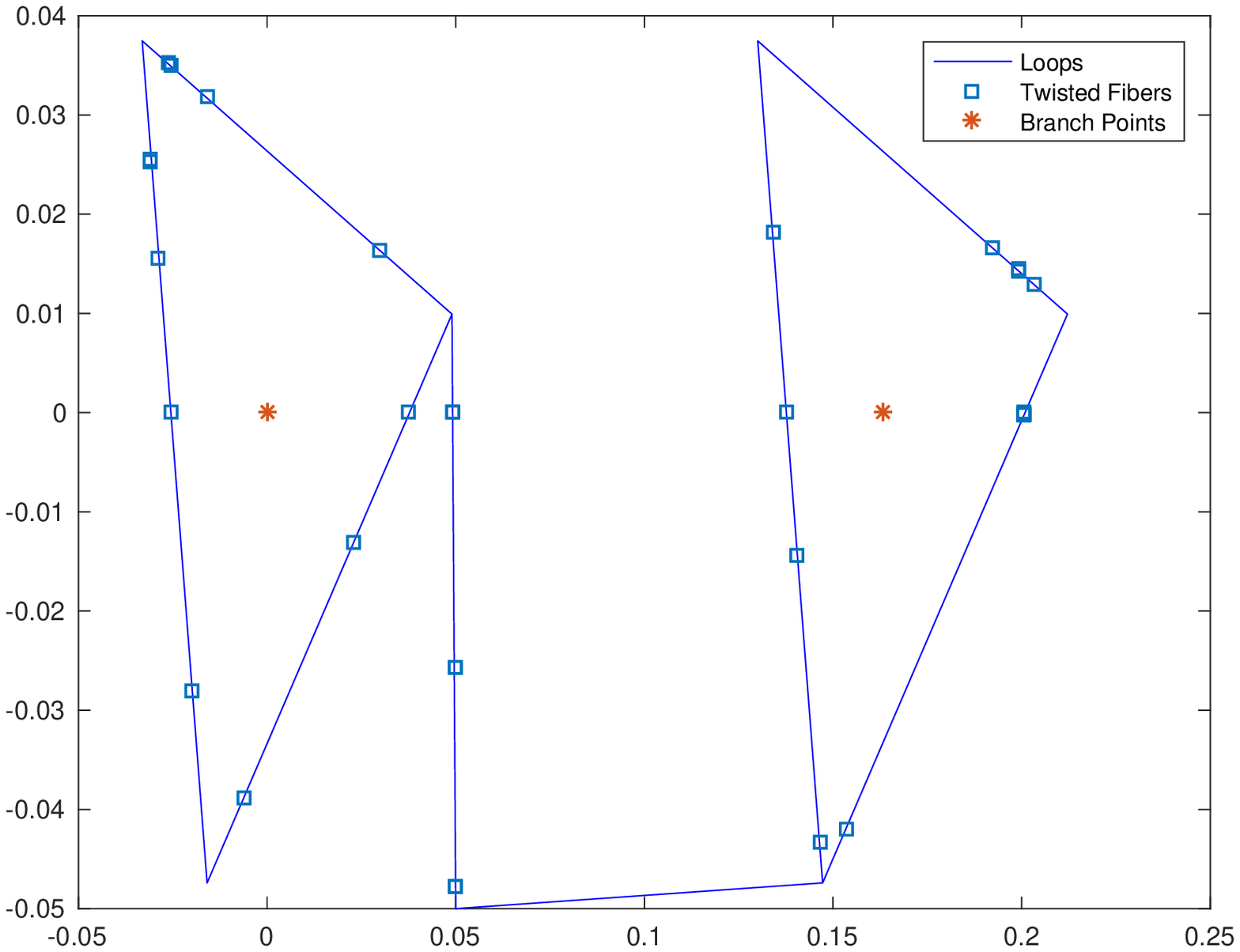} 
\includegraphics[scale=0.37]{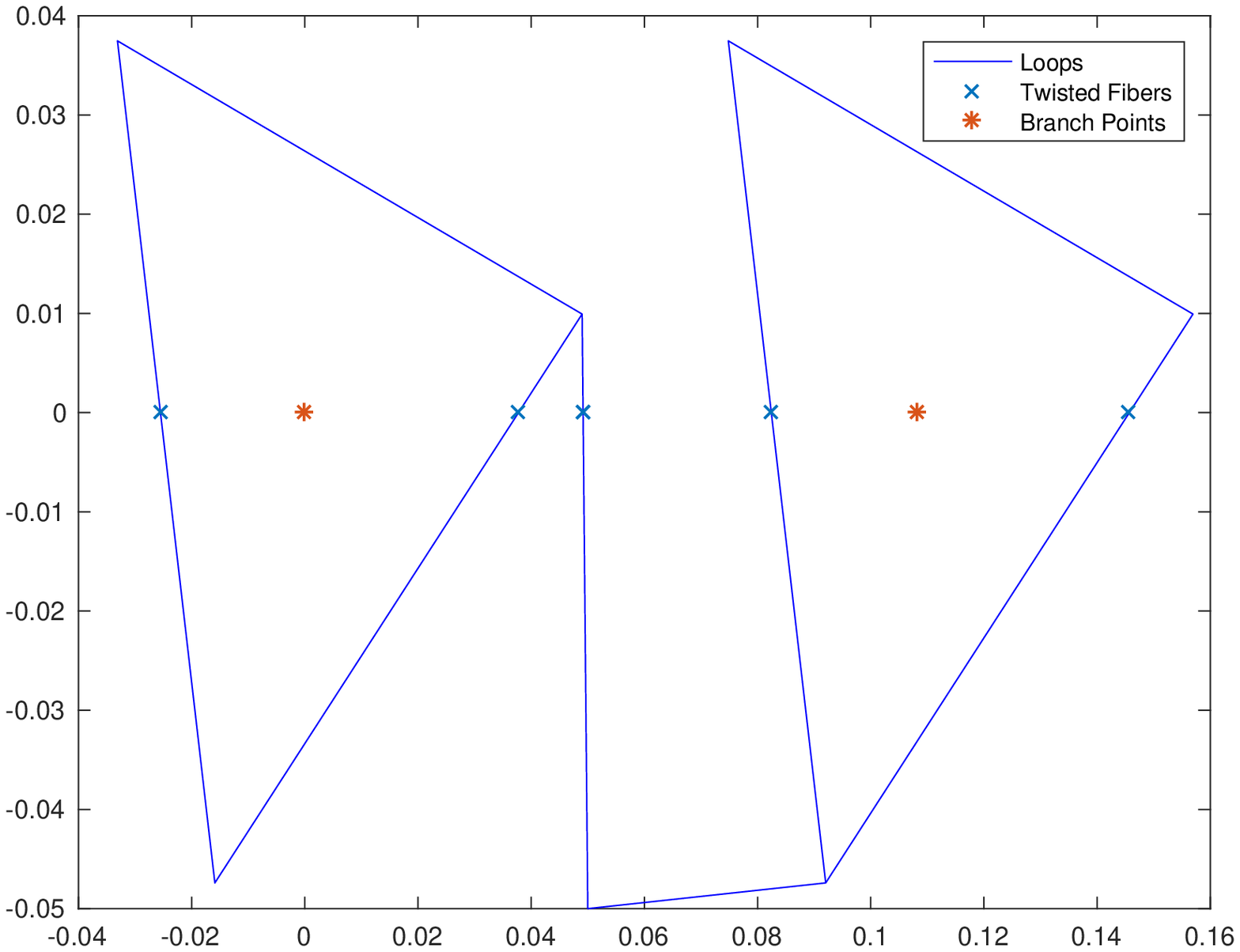}
\caption{These are the loops we tracked and points where we had crossings in the twists 
for $\mathcal{D}_i$. }\label{d12}
\end{figure}

\pagebreak
%%%EXAMPLE 4

\begin{example}\label{ex:notGen}
The braid group carries more information than the monodromy group of a curve. 
For example, consider the curve $\mathcal{C}$ defined by 
$$f(z,t):=z^3-t^2(1-t).$$
There are two branch points of this curve with respective braids below.
$$
\begin{array}{ccc}
\begin{array}{c}
\text{ Branch Point: }t=1\\
\hline
\begin{tikzpicture}
\braid[number of strands=3,height=.5in] (braid) 
a_2^{-1} a_1^{-1};
\end{tikzpicture}
\end{array}
&
\begin{array}{c}
\text{ Branch Point: }t=0\\
\hline
\begin{tikzpicture}
\braid[number of strands=3,height=.25in] (braid) 
a_2^{-1} a_1^{-1} a_2^{-1} a_1^{-1};
\end{tikzpicture}
\end{array}
\end{array}
$$

\end{example}
In Example \ref{ex:notGen}, we see that
the monodromy group of this curve is a cyclic group of three elements and is generated by a loop encircling any one of the branch points. 
The braid group on the other hand, is generated by the loop encircling the branch point at $t=1$ and is 
\emph{not} generated by the loop encircling the branch point at $t=0$.
Thus, a set of loops generating the monodromy group does not necessarily yield a set of loops that generate the braid group.

%{\color{red} Can we put something interesting here?If we restrict to a general line in the parameter space and pull that back to a curve and project this curve to $\mathbb{C}_z\times\mathbb{C}_t$ is the braid group preserved with the right hypothesis? Yes. Botong thought this fact is used in the Galois group paper. This follows some version of Lefschetz theorem.}

\section{Braid groups of hyperplane arrangements}\label{sec:hyper}

One is interested in braid groups of reducible varieties and in particular hyperplane arrangements. 
As we mentioned in Remark \ref{badcrossing}, care must take care when the fiber is not proper less we do some type of regularization procedure, which is not the focus of this article. 
Here is an example for a hyperplane arrangement with proper fibers.

\begin{example}\label{ex:hyper8}
%{\color{red}
%Botong, can you choose an example that is interesting for me to compute for us. Botong does not have a particular interesting example. How about the Example 10.5 in https://arxiv.org/pdf/math/0010105.pdf (with a general rotation.)}
We consider the following hyperplane arrangement in $\mathbb{P}^2$ 
$$twz(t-w)(t+w+z)=0$$
and consider a general projection to $\mathbb{P}^1$.
The branch locus has eight points. These consists of the pairwise intersections of the lines. 
The twists for these respective branch points are as follows.

\[
\begin{array}{cccccccccccccccccccc}
\begin{array}{c}  %%%BRANCH 1
%t=-1.6588+3.3869\sqrt{-1}\\
\hline
\begin{tikzpicture} 
\braid[number of strands=5,height=.1in,width=.1in,
       style strands={}{draw=none}] (braid)  a_4 a_3 a_4 a_2 a_3 a_4 a_4 a_3^{-1} a_2^{-1} a_4 a_3^{-1} a_4^{-1};
\end{tikzpicture}
\end{array}
&\quad
\begin{array}{c} %%%BRANCH2
%t=\\
\hline
\begin{tikzpicture} 
\braid[number of strands=5,height=.1in,width=.1in,
       style strands={}{draw=none}]] (braid)  a_3^{-1} a_4^{-1} a_2^{-1} a_3^{-1} a_1^{1}
       a_2^{-1} a_3^{-1}  a_1^{-1} a_4 a_4 a_1 a_3 a_2 a_1 a_3 a_2 a_4 a_3;
\end{tikzpicture}
\end{array}
&\quad
\begin{array}{c} %%%BRANCH3
%t=\\
\hline
\begin{tikzpicture} 
\braid[number of strands=5,height=.2in,width=.1in,
       style strands={}{draw=none}]] (braid)  a_3^{-1}  a_4^{-1} a_2^{-1} a_3^{-1} a_3  a_2 a_4 a_3;
\end{tikzpicture} 
\end{array}
&\quad
\begin{array}{c} %%%BRANCH4
%t=\\
\hline
\begin{tikzpicture} 
\braid[number of strands=5,height=.12in,width=.1in,
       style strands={}{draw=none}]] (braid)  a_3^{-1}  a_4^{-1} a_2^{-1}
       a_3^{-1} a_1^{-1} a_2^{-1} a_1^{-1} a_3 a_1 a_3 a_2 a_1 a_3 a_2 a_4 a_3;
\end{tikzpicture} 
\end{array}
%\\
&
%\\
\begin{array}{c} %%%BRANCH5
%t=\\
\hline
\begin{tikzpicture} 
\braid[number of strands=5,height=.3in,width=.1in,
       style strands={}{draw=none}]] (braid)  a_3^{-1} a_4 a_2 a_2 a_4^{-1} a_3;
\end{tikzpicture} 
\end{array}
&\quad
\begin{array}{c} %%%BRANCH6
%t=\\
\hline
\begin{tikzpicture} 
\braid[number of strands=5,height=.9in,width=.1in,
       style strands={}{draw=none}]] (braid)  a_3 a_3;
\end{tikzpicture} 
\end{array}
&\quad
\begin{array}{c} %%%BRANCH7
%t=\\
\hline
\begin{tikzpicture} 
\braid[number of strands=5,height=.24in,width=.1in,
       style strands={}{draw=none}]] (braid)  a_3^{-1}  a_4^{-1} a_2^{-1}       a_3 a_3 a_2 a_4 a_3;
\end{tikzpicture} 
\end{array}
&\quad
\begin{array}{c} %%%BRANCH8
%t=\\
\hline
\begin{tikzpicture} 
\braid[number of strands=5,height=.5in,width=.1in,
       style strands={}{draw=none}]] (braid)  a_3^{-1} a_4 a_4 a_3;
\end{tikzpicture} 
\end{array}
\end{array}
\]

\end{example}

The branch locus of a hyperplane arrangement is the projection of the finite set of union of pairwise intersections of the lines. 
There is at most $\binom{d}{2}$ branch points where $d$ is the number of hyperplanes in the arrangement. 
The cross-locus is determined by consider pairs of hyperplanes. 
This makes the polynomial system much easier to solve. Rather than solving a nonlinear system with Bezout bound of $d^4$ over each segment, 
we determine the cross-locus by solving $\binom{d}{2}$  systems of quadratics with Bezout bound of $2^4$.
This approach is much better when $d\gg 2$. 
Indeed, this leads us to develop Algorithm \ref{algo:braidHyper}.

\begin{algorithm}[hbt!]
\DontPrintSemicolon % Some LaTeX compilers require you to use \dontprintsemicolon instead
\KwIn{
\begin{enumerate}
\item $h_1(x,t),\dots,h_d(x,t)$ defining the hyperplanes $H_1,\dots,H_d$ of the arrangement. 
\end{enumerate}
}
\KwOut{A set of generators of the braid group of the hyperplane arrangement.}
Set $R$ to be
$\bigcup_{1\leq i_1<i_2\leq d}\,\pi(H_{i_1}\cap H_{i_2})$
with multiplicities removed.\;
Set $T$ to be the empty set, and
set $\bf{B}$ to be the empty set.\;
\For{ Branch point $\tau$ in $R$ } { 
  Construct a loop $\gamma$ encircling $\tau$ such that it is a concatenation of line segments $G$. \;
  \For{ $\Ldelta$ in $G$}{
    \For{ $1\leq i_1<i_2\leq d$}{
    Solve equations \eqref{eq:twistLocus} with $f=h_{i_1}h_{i_2}$ to 
    compute a finite set containing  $\Cross(\Ldelta)$.  \;
  Append  $\Cross(\Ldelta)$ to $T$.}}
Set $\bf{b}$ to be the identity element of the braid group.\;  r
\For{ $p$ in $T$}{
  Use homotopy continuation to track and determine the crossing. \;
  Append the cross to $\bf{b}$.\;
  \If{$\bf{b}\neq\emptyset$} {
  Append the twist $\bf{b}$ to $\bf{B}$.}    }}
\Return{{\bf B} }\;
\caption{Generators of the braid group of a hyperplane arrangement}
\label{algo:braidHyper}
\end{algorithm}

\begin{proof}[Proof of correctness for Algorithm \ref{algo:braidHyper}]
%{\color{red} What should go here?}
Based on Algorithm 1, the correctness for Algorithm follows from the following lemma. 
\end{proof}
\begin{lemma}
Under the notations of Proposition \ref{prop:eq}, suppose $f(z, t)$ factors as a product of polynomials $f(z, t)=f_1(z, t)\cdots f_l(z, t)$. Then the set of real solutions $(s, x, y_1, y_2)$ of equations (\ref{eq:twistLocus}) is equal to the union of the real solutions of 
\begin{equation}\label{eq:ij}
\begin{aligned}
f_{ij}\left(x+\sqrt{-1}y_1,(1-s)\delta'+s\delta''\right)=0,    
&&  g_{ij}\left(x-\sqrt{-1}y_1,(1-s)\bar\delta'+s\bar\delta''\right)=0\\
 f_{ij}\left(x+\sqrt{-1}y_2,(1-s)\delta'+s\delta''\right)=0,   
 &&	g_{ij}\left(x-\sqrt{-1}y_2,(1-s)\bar\delta'+s\bar\delta''\right)=0
\end{aligned}
\end{equation}
for all $1\leq i<j\leq l$, where $f_{ij}(z, t)=f_i(z, t)f_j(z, t)$ and $g_{ij}$ is the polynomial whose coefficients are the conjugates of the coefficients of $f_{ij}$. 
\end{lemma}
\begin{proof}
The real solutions of equations (\ref{eq:ij}) are real solutions of equation (\ref{eq:twistLocus}). Conversely, suppose $(s, x, y_1, y_2)$ is a solution of equations (\ref{eq:twistLocus}). Then there exists $i$ and $j$ such that
$$f_{i}\left(x+\sqrt{-1}y_1,(1-s)\delta'+s\delta''\right)= f_{j}\left(x+\sqrt{-1}y_2,(1-s)\delta'+s\delta''\right)=0$$
for some $1\leq i, j\leq l$. Without loss of generality, we assume that $i\leq j$. Thus, $(s, x, y_1, y_2)$ is a solution to the equations in equations (\ref{eq:ij}) involving $f_{ij}$. Notice that when assuming $s, x, y_1, y_2$ are real, the equations in (\ref{eq:ij}) involving $g_{ij}$ become redundant. Thus, the lemma follows. 
\end{proof}

\pagebreak

\begin{example}\label{ex:hyper46}
In the following example we use Algorithm \ref{algo:braidHyper}.
This example is a hyperplane arrangement given by fourteen lines. These lines are defined by the following equations: 

$$
[w_0,w_1,w_2]  \left[\begin{array}{cccccccccccccc}
1 & 0 & 1 & 0 & 1 & 1 & 2 & 3 & 2 & 3 & 1 & 1 & -1 & 0\\
0 & 1 & 0 & 1 & 1 & 2 & 1 & 2 & 3 & 1 & 3 & -1 & 1 & 0\\
0 & 0 & 1 & 1 & 2 & 2 & 2 & 4 & 4 & 4 & 4 & 4 & 4 & 1
\end{array}\right]=0.
$$
This hyperplane arrangement has 46 branch points, which together have over $2568$ crossings. 
Two of the the 46 generators are below. 

$$
\begin{array}{c} 
{\bf b}_1=g\sigma_{5}  \sigma_{5}g^{-1} \text{ with }\\
\begin{array}{lll}
g&:= &\big(\sigma_{3}  \sigma_{4}  \sigma_{3}  \sigma_{12}^{-1}  \sigma_{11}^{-1}  \sigma_{12}^{-1}  \sigma_{2}  \sigma_{10}  \sigma_{9} 
 \sigma_{8}  \sigma_{1}  \sigma_{13}  \sigma_{7}  \sigma_{6}  \sigma_{5}  \sigma_{12}  \sigma_{4}  \sigma_{11}  \sigma_{10}
  \sigma_{3}  \sigma_{2}  \sigma_{9}  \sigma_{11}  \sigma_{3}  \sigma_{8} 
 \\
 &&
 \sigma_{7}  \sigma_{6}  \sigma_{5}  \sigma_{4}  \sigma_{3}  \sigma_{5}  \sigma_{8}  \sigma_{7}  \sigma_{8}  \sigma_{6}  \sigma_{7}  \sigma_{8}  \sigma_{10}  \sigma_{9} 
 \sigma_{10}  \sigma_{12}  \sigma_{11}  \sigma_{12}  \sigma_{8}^{-1}  \sigma_{1}^{-1}  \sigma_{2}^{-1}  \sigma_{1}^{-1}  \sigma_{7}^{-1}  \sigma_{9}^{-1}   \big)\\
\end{array}
\\
\hline
\begin{tikzpicture} 
\braid[number of strands=14,height=.05in,width=.35in,
       style strands={}{draw=none}] (braid)  
a_{9}  a_{7}  a_{1}  a_{2}  a_{1}  a_{8}  a_{12}^{-1}  a_{11}^{-1}  a_{12}^{-1}  a_{10}^{-1}  a_{9}^{-1}         a_{10}^{-1}  a_{8}^{-1}  a_{7}^{-1}  a_{6}^{-1}  a_{8}^{-1}  a_{7}^{-1}  a_{8}^{-1}  a_{5}^{-1}  a_{3}^{-1}  a_{4}^{-1}  a_{5}^{-1}  a_{6}^{-1}  a_{7}^{-1}  a_{8}^{-1}  a_{3}^{-1}  a_{11}^{-1}  a_{9}^{-1}  a_{2}^{-1}  a_{3}^{-1}  a_{10}^{-1}  a_{11}^{-1}  a_{4}^{-1}  a_{12}^{-1}  a_{5}^{-1}  a_{6}^{-1}  a_{7}^{-1}  a_{13}^{-1}  a_{1}^{-1}  a_{8}^{-1}  a_{9}^{-1}  a_{10}^{-1}  a_{2}^{-1}  a_{12}  a_{11}  a_{12}  a_{3}^{-1}  a_{4}^{-1}  a_{3}^{-1}  a_{5}  a_{5}  a_{3}  a_{4}  a_{3}  a_{12}^{-1}  a_{11}^{-1}  a_{12}^{-1}  a_{2}  a_{10}  a_{9}  a_{8}  a_{1}  a_{13}  a_{7}  a_{6}  a_{5}  a_{12}  a_{4}  a_{11}  a_{10}  a_{3}  a_{2}  a_{9}  a_{11}  a_{3}  a_{8}  a_{7}  a_{6}  a_{5}  a_{4}  a_{3}  a_{5}  a_{8}  a_{7}  a_{8}  a_{6}  a_{7}  a_{8}  a_{10}  a_{9}  a_{10}  a_{12}  a_{11}  a_{12}  a_{8}^{-1}  a_{1}^{-1}  a_{2}^{-1}  a_{1}^{-1}  a_{7}^{-1}  a_{9}^{-1}   ; 
\node [at=(braid-1-e),pin=south  west :   1]  {};
\node [at=(braid-2-e),pin=south  west :   2]  {};
\node [at=(braid-3-e),pin=south  west :   3]  {};
\node [at=(braid-4-e),pin=south  west :   4]  {};
\node [at=(braid-5-e),pin=south  west :   5]  {};
\node [at=(braid-6-e),pin=south  west :   6]  {};
\node [at=(braid-7-e),pin=south  west :   7]  {};
\node [at=(braid-8-e),pin=south  west :   8]  {};
\node [at=(braid-9-e),pin=south  west :   9]  {};
\node [at=(braid-10-e),pin=south  west :   10]  {};
\node [at=(braid-11-e),pin=south  west :   11]  {};
\node [at=(braid-12-e),pin=south  west :   12]  {};
\node [at=(braid-13-e),pin=south  west :   13]  {};
\node [at=(braid-14-e),pin=south  west :   14]  {};
\end{tikzpicture} \\
\end{array}
$$

$$
\begin{array}{c} 
%t=\\
{\bf b}_2=\left(\sigma_{9}  \sigma_{10}\right)^{-1} \sigma_{9}  \sigma_{10}  \sigma_{9}  \sigma_{10}  \sigma_{9}  \sigma_{10}  \left(\sigma_{9}  \sigma_{10}  \right) \\
\hline
\begin{tikzpicture} 
\braid[number of strands=14,height=.1in,width=.35in,
       style strands={}{draw=none}] (braid)  
a_{10}^{-1}  a_{9}^{-1}  a_{9}  a_{10}  a_{9}  a_{10}  a_{9}  a_{10}  a_{9}  a_{10}   ; 
  ; 
\node [at=(braid-1-e),pin=south  west :   1]  {};
\node [at=(braid-2-e),pin=south  west :   2]  {};
\node [at=(braid-3-e),pin=south  west :   3]  {};
\node [at=(braid-4-e),pin=south  west :   4]  {};
\node [at=(braid-5-e),pin=south  west :   5]  {};
\node [at=(braid-6-e),pin=south  west :   6]  {};
\node [at=(braid-7-e),pin=south  west :   7]  {};
\node [at=(braid-8-e),pin=south  west :   8]  {};
\node [at=(braid-9-e),pin=south  west :   9]  {};
\node [at=(braid-10-e),pin=south  west :   10]  {};
\node [at=(braid-11-e),pin=south  west :   11]  {};
\node [at=(braid-12-e),pin=south  west :   12]  {};
\node [at=(braid-13-e),pin=south  west :   13]  {};
\node [at=(braid-14-e),pin=south  west :   14]  {};
\end{tikzpicture} 
\end{array}
$$

\end{example}

\section{Braid groups of polynomial systems}\label{sec:poly}
In this section we consider braid groups of polynomial systems that define a variety in $
\mathbb{P}^m\times\mathbb{P}^n$. We first consider the case where $m=1$.

Let $\hat X$ denote a hypersurface in $\mathbb{P}^1_z\times\mathbb{P}_u^m$ such that each irreducible component of the hypersurface has a dominant projection to $\mathbb{P}_u^m$. 
Let $X$ be the restriction of $\hat X$ to a product of general affine charts $\mathbb{C}_z^1\times\mathbb{C}_u^m$. 
Then, $X$ is a hypersurface and suppose this hypersurface is defined by the polynomial $f(z,u)=0$.

\begin{definition}
The branch locus $\branch$  of $\pi:X\to\mathbb{C}_u^m$ 
is defined to be 
\begin{equation}\left\{
u\in\mathbb{C}^m_u : \text{there exists } z\text{ such that } f(z,u)=0,\frac{\partial f}{\partial z}=0
\right\}.
\end{equation}
\end{definition}

Let $\gamma$ denote a loop in $\mathbb{P}^m_u\setminus\branch$. If this loop is a concatenation of line segments, then item \ref{item:important} of Algorithm \ref{algo:braid} 
allows us to compute a generator of the braid group. 
If we were to take a finite set of loops then we would have computed a subgroup of the braid group of the projection. If this subgroup is actually the full braid group on $\deg\pi$ strands then one would have computed the braid group. 
This method is useful because it can be run without computing the branch locus and taking sufficiently general  loops $\gamma$ in   $\mathbb{C}_u^m$.

To compute the braid group when it is a proper subgroup of the full braid group on 
$\deg\pi$ strands, one restricts $\mathbb{C}_u^m$ to a general line. 
The curve $\mathcal{C}$ induced by the inverse image $\pi^{-1}$ is a subvariety of $X$. 
If we parameterize the line by $\mathbb{C}_t^1$, 
then we may consider the curve $\mathcal{C}$ as a curve of  
$\mathbb{C}^1_z\times\mathbb{C}^1_t$.
The braid group of the projection of the  curve $\mathcal{C}$ to $\mathbb{C}_t^1$ is the braid group of $\pi:X\to\mathbb{C}_u^m$ if the line parameterized by $\mathbb{C}_t$ is general. 
This follows from Zariski's theorem \cite{Zar}. %{\color{red}Is this Lefschetz theorem or Zariski theorem?}

If $Y$ is an $n$ dimensional variety of $\mathbb{C}_y^m\times\mathbb{C}^n_u$, then we consider a projection $\mu:\mathbb{C}_y^m\to\mathbb{C}_z^1$. 
Let $X_\mu$ denote induced image of $\mu:Y\to\mathbb{C}_z^1\times\mathbb{C}_u^n$.
If each irreducible component of $X_\mu$ has a dominant projection to $\mathbb{C}_u^n$ then we have the braid group for $X_\mu$.  Also suppose that the  $\deg \pi:X_\mu\to\mathbb{C}_u^n$  equals 
$\deg \pi: Y\to\mathbb{C}_u^m$.
There are several interesting projections one may consider including coordinate projections and general projections. 
%When the projection $\mu$ is sufficiently general then the braid group of $X_\mu\to\mathbb{C}_u^m$ and the braid group of $Y\to\mathbb{C}_u^m$ coincide. 
%{\color{red}Botong doesn't know the definition of a braid group when the subvariety $Y\in \mathbb{C}_y^m\times\mathbb{C}^n_u$ is not of codimension one. Just imagine 1-dimensional strands can move freely in a space of real dimension $\geq$ 4. So the above definition is kind of artificial. The last sentence should really be a definition instead of a theorem. }
%This follows by \ref{}. 
One of the advantageous of our numerical methods is that we do not need to do the symbolic computation of elimination to determine the defining equation of the hypersurface $X_\mu$.
Instead we can work with the variety $Y$ and order the fiber using a particular coordinate or linear combination of the coordinates.

\section{Conclusion}
We have provided a framework to numerically compute braid groups of curves, hyperplane arrangements, and parameterized polynomial systems. 
This is done by developing Algorithms \ref{algo:braid} and \ref{algo:braidHyper}.
Moreover, we have implemented these algorithms through Macaulay2 \cite{BMonodromy} and Bertini. 
There are many directions that future work can take. 
For example, further work should be explored in terms of certified path tracking. 
Also rather than computing the twist locus via polynomial system solving we described, one could use interval arithmetic 
to keep track of when the strands cross. 
In Section \ref{sec:hyper}, we designed a specialized algorithm for line arrangements and saw huge performance increases both theoretically and in implementation.
Further work on tailoring algorithms for  interesting families  of problems is also of interest.

%For example, software by \ref{Clemson} is used. 

\section*{Acknowledgments}
The authors thank Jordan Ellenberg for helpful conversations. 
The first author is supported by  a University of Chicago Provost Postdoctoral Scholarship.
The second author is partially supported by NSF grant DMS-1701305.

\bibliographystyle{abbrv}
\bibliography{refs_bg}

\begin{thebibliography}{10}

\bibitem{allgower2012numerical}
E.~L. Allgower and K.~Georg.
\newblock {\em Numerical continuation methods}, volume~13 of {\em Springer
  Series in Computational Mathematics}.
\newblock Springer-Verlag, Berlin, 1990.
\newblock An introduction.

\bibitem{allgower1993continuation}
E.~L. Allgower and K.~Georg.
\newblock Continuation and path following.
\newblock In {\em Acta numerica, 1993}, Acta Numer., pages 1--64. Cambridge
  Univ. Press, Cambridge, 1993.

\bibitem{bertini4M2}
D.~Bates, E.~Gross, A.~Leykin, and J.~Rodriguez.
\newblock {Bertini for Macaulay2}.
\newblock {\tt arXiv/org:1603.05908}, 2013.

\bibitem{Bertini}
D.~J. Bates, J.~D. Hauenstein, A.~J. Sommese, and C.~W. Wampler.
\newblock Bertini: Software for numerical algebraic geometry.
\newblock Available at https://bertini.nd.edu/.

\bibitem{CKLS2001}
J.~C. Cha, K.~H. Ko, S.~J. Lee, J.~W. Han, and J.~H. Cheon.
\newblock An efficient implementation of braid groups.
\newblock In {\em Advances in cryptology---{ASIACRYPT} 2001 ({G}old {C}oast)},
  volume 2248 of {\em Lecture Notes in Comput. Sci.}, pages 144--156. Springer,
  Berlin, 2001.

\bibitem{Dehornoy2004}
P.~Dehornoy.
\newblock Braid-based cryptography.
\newblock In {\em Group theory, statistics, and cryptography}, volume 360 of
  {\em Contemp. Math.}, pages 5--33. Amer. Math. Soc., Providence, RI, 2004.

\bibitem{BMonodromy}
T.~Duff, C.~Hill, A.~Jensen, K.~Lee, A.~Leykin, and J.~Sommars.
\newblock Solving polynomial systems via homotopy continuation and monodromy.
\newblock {\em arXiv:1609.08722}, 2017.

\bibitem{GP11}
A.~Galligo and A.~Poteaux.
\newblock Computing monodromy via continuation methods on random {R}iemann
  surfaces.
\newblock {\em Theoret. Comput. Sci.}, 412(16):1492--1507, 2011.

\bibitem{Garber2010}
D.~Garber.
\newblock Braid group cryptography.
\newblock In {\em Braids}, volume~19 of {\em Lect. Notes Ser. Inst. Math. Sci.
  Natl. Univ. Singap.}, pages 329--403. World Sci. Publ., Hackensack, NJ, 2010.

\bibitem{HRS17}
J.~D. Hauenstein, J.~I. Rodriguez, and F.~Sottile.
\newblock Numerical computation of galois groups.
\newblock {\em Foundations of Computational Mathematics}, Jun 2017.

\bibitem{KLC2000}
K.~H. Ko, S.~J. Lee, J.~H. Cheon, J.~W. Han, J.-s. Kang, and C.~Park.
\newblock New public-key cryptosystem using braid groups.
\newblock In {\em Advances in cryptology---{CRYPTO} 2000 ({S}anta {B}arbara,
  {CA})}, volume 1880 of {\em Lecture Notes in Comput. Sci.}, pages 166--183.
  Springer, Berlin, 2000.

\bibitem{Kur2012}
V.~Kurlin.
\newblock Computing braid groups of graphs with applications to robot motion
  planning.
\newblock {\em Homology Homotopy Appl.}, 14(1):159--180, 2012.

\bibitem{LLH2001}
E.~Lee, S.~J. Lee, and S.~G. Hahn.
\newblock Pseudorandomness from braid groups.
\newblock In {\em Advances in cryptology---{CRYPTO} 2001 ({S}anta {B}arbara,
  {CA})}, volume 2139 of {\em Lecture Notes in Comput. Sci.}, pages 486--502.
  Springer, Berlin, 2001.

\bibitem{LS09}
A.~Leykin and F.~Sottile.
\newblock Galois groups of {S}chubert problems via homotopy computation.
\newblock {\em Math. Comp.}, 78(267):1749--1765, 2009.

\bibitem{li_1997}
T.~Y. Li.
\newblock Numerical solution of multivariate polynomial systems by homotopy
  continuation methods.
\newblock In {\em Acta numerica, 1997}, volume~6 of {\em Acta Numer.}, pages
  399--436. Cambridge Univ. Press, Cambridge, 1997.

\bibitem{Lib}
A.~Libgober.
\newblock Invariants of plane algebraic curves via representations of the braid
  groups.
\newblock {\em Invent. Math.}, 95(1):25--30, 1989.

\bibitem{MNMH}
D.~Molzahn, M.~Niemerg, D.~Mehta, and J.~Hauenstein.
\newblock Investigating the maximum number of real solutions to the power flow
  equations: Analysis of lossless four-bus systems.
\newblock {\tt arXiv/org:1603.05908}, 2016.

\bibitem{NW1996}
C.~Nayak and F.~Wilczek.
\newblock 2n-quasihole states realize 2n−1-dimensional spinor braiding
  statistics in paired quantum hall states.
\newblock {\em Nuclear Physics B}, 479(3):529 -- 553, 1996.

\bibitem{Po07}
A.~Poteaux.
\newblock Computing monodromy groups defined by plane algebraic curves.
\newblock In {\em S{NC}'07}, pages 36--45. ACM, New York, 2007.

\bibitem{PHC}
J.~Verschelde.
\newblock Polynomial homotopy continuation with phcpack.
\newblock {\em ACM Commun. Comput. Algebra}, 44(3/4):217--220, Jan. 2011.

\bibitem{Zar}
O.~Zariski.
\newblock A theorem on the {P}oincar\'e group of an algebraic hypersurface.
\newblock {\em Annals of Mathematics}, 38(1):131--141, 1937.

\end{thebibliography}

\end{document}